\documentclass[a4paper,11pt]{article}

\usepackage{amsmath,amssymb,amsthm}
\usepackage[margin=2cm]{geometry}
\usepackage{color}

\usepackage[utf8]{inputenc}
\usepackage[T1]{fontenc}
\usepackage{dsfont}
\usepackage{mathrsfs}
\usepackage{hyperref}

\newcommand{\bbR}{\mathbb{R}}

\renewcommand{\phi}{\varphi}

\newtheorem{theo}{Theorem}[section]
\newtheorem{lemm}[theo]{Lemma}
\newtheorem{prop}[theo]{Proposition}

\newtheorem{rema}[theo]{Remark}

\newtheorem{coro}[theo]{Corollary}

\title{Narrow equidistribution and counting of closed geodesics on noncompact manifolds}
\author{Barbara Schapira, Samuel Tapie}
\date{Version \today}

\begin{document}   
\maketitle

\begin{abstract}  We prove the equidistribution of (weighted) periodic orbits of the geodesic flow on noncompact negatively curved manifolds toward equilibrium states in the narrow topology, i.e. in the dual of bounded continuous functions. We deduce an exact asymptotic counting for periodic orbits (weighted or not), which was previously known only for geometrically finite manifolds. 
\footnote{ Keywords\,: Negative curvature, geodesic flow, periodic orbits, equidistribution, Gibbs measure, counting.}
\footnote{MSC Classification   37A25, 37A35, 37D35, 37D40.}
\end{abstract}

\section{Introduction}

A well known feature of compact hyperbolic dynamics is the abundance of periodic orbits: they have a positive exponential growth rate, equal to the topological entropy. 
Moreover, the periodic measures supported by these orbits become equidistributed towards the measure of maximal entropy of the system. 
A weighted version of this property also holds : given any H\"older potential $F$, the periodic measures weighted by the periods of the potential become equidistributed towards the (unique) equilibrium state of the potential, see the classical works of Bowen \cite{Bowen} and Parry-Pollicott \cite{PP}. 

A typical geometric example  is the geodesic flow of a compact negatively curved manifold, which is an Anosov flow, and therefore satisfies the above properties. In this geometric context, it has been proved in \cite{Roblin, PPS} that similar equidistribution properties also hold in a noncompact setting. Let $M$ be a negatively curved manifold, $T^1M$ its unit tangent bundle, and $(g^t)$ the geodesic flow. As soon as it admits a finite invariant measure maximizing entropy, called the {\em 
Bowen-Margulis-Sullivan measure} $m_{BMS}$, 
the average of all orbital measures supported by periodic orbits of length at most $T$ converge to the normalized measure $\bar m_{BMS}$, in 
the {\em vague topology}, i.e. in the dual of continuous functions \emph{with compact support}.

\medskip

A  weighted version of this result also holds\,:  given a H\"older continuous map $F:T^1M\to \mathbb{R}$, 
as soon as it admits a finite equilibrium state $m_F$ the orbital measures supported by periodic orbits of length at most $T$, conveniently weighted by the periods of the potential $F$, 
converge to the normalized measure $\bar m_F$ in the vague topology.

\medskip

A major motivation for proving such equidistribution results is to get asymptotic counting estimates for the number of (weighted) periodic orbits of length at most $T$. However, it turns out that the equidistribution property required to get such counting estimates is a stronger convergence, in the {\em narrow topology}, i.e. the dual of
continuous \emph{bounded} functions. Until now, such \emph{narrow equidistribution} or such \emph{asymptotic counting} for periodic orbits have been proven only when $M$ is {\em geometrically finite} in \cite{Roblin, PPS}. 

In this note, inspired by work done in \cite{PS16} and \cite{ST19}, we remove this assumption and show this narrow equidistribution as soon as the Bowen-Margulis-Sullivan (or the equilibrium state) is finite.

\medskip

Let us precise some notations, mostly coming from \cite{PPS}. 
We denote by $\mathcal{P}'$ the set of primitive periodic orbits of the geodesic flow and $\mathcal{P}'(t)$ the subset of those primitive periodic orbits with length $\ell(p)$ at most $t$. Given a compact set $\mathcal{W}\subset T^1M$, we denote by $\mathcal{P}'_\mathcal{W}$ (resp.  $\mathcal{P}'_\mathcal{W}(t)$)
 the set of primitive periodic orbits (resp. of length at most $t$) which intersect $\mathcal W$. 
If $p$ is an oriented periodic orbit, let $\mathcal{L}_p$ be the periodic measure of mass $\ell(p)$ supported on $p$, so that $\frac{1}{\ell(p)}\mathcal{L}_p$ is a probability measure.

If $F:T^1M\to\bbR$ is a H\"older continuous potential, we denote by $\delta_F$ its critical exponent, and  
$m_F$ the Gibbs measure asociated to $F$, given by the Patterson-Sullivan-Gibbs construction (see \cite{PPS}). 
Under some additional geometric assumptions (pinched negative curvature, bounded derivatives of the curvature), one knows that  $\delta_F$ is also 
the topological pressure of $F$ (see \cite{OP} when $F=0$ and \cite{PPS} for general $F$), and that when $m_F$ is finite, the normalized probability measure $\bar m_F = m_F/\|m_F\|$ is the unique equilibrium state for $F$. We will not need this characterization here. 
 
Our main result is the following.

\begin{theo}\label{theo:equid} Let $M$ be a manifold with  negative curvature satisfying $\kappa\le -a^2<0$, whose geodesic flow is topologically mixing. 
Let $F:T^1M \to \bbR$ be a H\"older-continuous map, with finite critical exponent $\delta_F$, which admits a finite  Gibbs measure $m_F$. 
Assume without loss of generality that its topological pressure $\delta_F$ is positive. Let $W\subset M$ be a compact set and $\mathcal{W}=T^1W$. 
Assume that the interior of $\mathcal{W}$ intersects at least a periodic orbit of $(g^t)$. 
Then 
\begin{equation}\label{equation:equid}
\delta_F T  e^{-\delta_F T} \sum_{p\in \mathcal{P}'_\mathcal{W}(T)}e^{\int_p F} \frac{1}{\ell(p)}\mathcal{L}_p \to \frac{m_F}{\|m_F\|}\,  \quad\mbox{when}\quad T\to + \infty
\end{equation} in the narrow topology, i.e. in the dual of continuous bounded functions.
\end{theo}

Integrating the constant map equal to $1$ gives the following corollary. 
\begin{coro}\label{coro:gibbs} Under the same assumptions, we have 
\begin{equation}\label{equation:couting}  \sum_{p\in \mathcal{P}'_\mathcal{W}(T)}e^{\int_p F}  \sim \frac{e^{\delta_F T}}{\delta_F T}\,  \quad\mbox{when}\quad T\to + \infty\,.
\end{equation} 
\end{coro}

When $F=0$,  the exponent $\delta_0$ is exactly the  critical 
exponent $\delta_\Gamma$ of the group $\Gamma$ acting on $\widetilde{M}$, and the Gibbs measure $m_0$ is known as the Bowen-Margulis-Sullivan measure. 
As already mentioned above, when the curvature is pinched negative and the derivatives of the curvature are bounded, it is also the topological entropy of the geodesic flow, see \cite{OP}. Since our main result with $F=0$ is valid in the more general geometric setting of $CAT(-1)$-metric spaces, we restate it in this context.

\begin{theo}\label{theo:CAT(-1)}  Let $X$ be a $CAT(-1)$-metric space, and $\Gamma$ a discrete group of isometries acting properly on $X$. 
Assume that the geodesic flow of $X/\Gamma$ is topologically mixing and admits a finite Bowen-Margulis-Sullivan measure. 
Let $W\subset M$ be a compact set and $\mathcal{W}=T^1W$. 
Assume that  its interior intersects at least a periodic orbit of $(g^t)$. 
Then 
\begin{equation}
\delta_\Gamma T  e^{-\delta_\Gamma T} \sum_{p\in \mathcal{P}'_\mathcal{W}(T)} \frac{1}{\ell(p)}\mathcal{L}_p \to \frac{m_{BMS}}{\|m_{BMS}\|}\,  \quad\mbox{when}\quad T\to + \infty
\end{equation} in the narrow topology. 
\end{theo}

As a corollary, integrating the constant map equal to $1$, we get the following striking consequence of our work.

\begin{coro}\label{coro:striking} Let $M$ be a manifold with pinched negative curvature or a quotient  of a $CAT(-1)$-space, whose geodesic flow is topologically mixing. 
Assume that there exists a (finite) measure of maximal entropy, or equivalently that the
Bowen-Margulis-Sullivan measure is finite. Let $W\subset M$ be a compact set and $\mathcal{W}=T^1W$. 
Assume that its interior intersects at least a periodic orbit of $(g^t)$. 
Then 
\begin{equation}\label{equation:couting-zero-potentiel}  \#\,\mathcal{P}'_\mathcal{W}(T)   \sim \frac{e^{\delta_\Gamma T}}{\delta_\Gamma T}\,  \quad\mbox{when}\quad T\to + \infty\,.
\end{equation} 
\end{coro}

The results above are completely new in the geometrically infinite setting. We refer to \cite{ST19} (resp. \cite{STdufutur}) for several classes of geometrically infinite manifolds (resp.  potentials) satisfying the so-called {\em Strongly Positively Recurrent (SPR) property}, which implies the finiteness of the Bowen-Margulis-Sullivan measure (resp. of the associated Gibbs measure). 
In particular, there are wide classes of examples of manifolds / potentials satisfying the assumptions of Theorems \ref{theo:equid} or \ref{theo:CAT(-1)} and their corollaries. 

The asymptotic counting given in our last corollary  is due to Margulis \cite{Margulis} on compact negatively curved manifolds. On geometrically finite spaces, it is due to \cite{Roblin} when $F=0$  in the  $CAT(-1)$-setting. For general potentials on geometrically finite manifolds with pinched negative curvature, it had been shown in \cite{PPS}. 

Theorem \ref{theo:equid} and Corollary \ref{coro:gibbs} are also announced in \cite{Velozo} under the (more restrictive) assumptions that $F$ is a H\"older potential satisfying the SPR property and converging to $0$ at infinity. It seems to us that his (interesting) approach cannot work for general potentials. 

\medskip

The restriction to $\mathcal{P}'_W$ instead of $\mathcal {P}'$ is intrinsic to the noncompact geometrically  infinite case. 
Indeed, except in the geometrically finite case, where both sets typically coincide for $W$ large enough (containing the compact part of the manifold), $\mathcal{P}'_W(T)$ is 
 a finite set, whereas $\mathcal{P}'(T)$ could easily often be infinite. 

The assumption of finiteness of the measure in Theorems \ref{theo:equid} and \ref{theo:CAT(-1)} is unavoidable. Indeed, it is proven in \cite{Roblin, PPS} that the sum over all periodic orbits of $\mathcal{P}'(t)$ vaguely converges to $0$, and the sums considered in both Theorems \ref{theo:equid} and \ref{theo:CAT(-1)} are obviously smaller, and therefore converge also vaguely to $0$. However, in some situations, it can happen that counting results similar to  Corollary \ref{coro:striking} with different asymptotics hold, as in \cite[Chapter 7]{Vidotto}. 
This infinite measure situation is still widely unexplored.


\section{Thermodynamical formalism of the geodesic flow}

\subsection{Negative curvature}

Let $M$ be a non compact manifold, with negative sectional curvature satisfying $\kappa\le -a^2<0$ everywhere. 
In the sequel, we assume $M$ be nonelementary, i.e. there are at least two distinct closed geodesics on $M$ (and therefore an infinity). 

Let  $\Gamma=\pi_1(M)$ be its fundamental group, $\widetilde M$ be its universal cover, and $\partial\widetilde M$ its boundary at infinity. 
Denote by $\pi:T^1\widetilde M\to \widetilde M$ the canonical projection, and $p_\Gamma:\widetilde M\to M$ the quotient map. 

The Busemann cocycle is defined on $\partial \widetilde M\times \widetilde M\times\widetilde M$ by
$$
\beta_\xi(x,y)=\lim_{z\to \xi} d(x,z)-d(y,z)\,.
$$

The unit tangent bundle $T^1\widetilde M$ of $\widetilde M$ is homeomorphic to 
$\partial^2\widetilde M:=\partial\tilde M\times \partial \widetilde M \setminus \mbox{Diagonal}$ through the 
well known {\em Hopf coordinates}\,:
$$
v\mapsto (v^-,v^+,\beta_{v^+}(o,\pi(v)))\,,
$$ 
where $o$ is an arbitrary fixed point chosen once for all. 

The geodesic flow on $T^1\widetilde M$ or $T^1M$ is denoted by $(g^t)_{t\in\mathbb{R}}$. In the above coordinates, it acts by translation on the real factor. 

The action of $\Gamma$ can be expressed in these coordinates as follows\,:
$$
\gamma(v^-,v^+,t)=(\gamma v^-,\gamma v^+,t+\beta_{v^+}(\gamma^{-1}o,o))\,.
$$ 

Therefore, any invariant measure $m$ under the geodesic flow on $T^1M$ can be lifted into an invariant measure $\tilde m$ which, in these coordinates, can be written 
$\tilde m=\mu\times dt$ where $\mu$ is a $\Gamma$-invariant measure on $\partial^2\widetilde M$ and $dt$ is the Lebesgue measure on $\bbR$. \\

As said in the introduction, we denote by 
\begin{itemize}
\item $\mathcal{P}$ (resp. $\mathcal{P}'$) the set of  periodic orbits (resp. primitive periodic orbits) of $(g^t)$ on $T^1M$;
\item  $\mathcal{P}(t)$ (resp. 
$\mathcal{P}'(t)$) the set of (primitive) periodic orbits of length at most $t$. 
\item $\mathcal{P}(t_1,t_2)$  (resp. 
$\mathcal{P}'(t_1,t_2)$) the set of (primitive) periodic orbits of length in $(t_1,t_2]$, 
\end{itemize}
When $M$ is noncompact, all these sets can be infinite. We therefore consider only periodic orbits intersecting a given compact set $\mathcal{W}$. 
We denote by $\mathcal{P}_\mathcal{W}, \mathcal{P}'_\mathcal{W}, \mathcal{P}_\mathcal{W}(t), \mathcal{P}'_\mathcal{W}(t),\mathcal{P}_\mathcal{W}(t_1,t_2), \mathcal{P}'_\mathcal{W}(t_1,t_2)$ 
the corresponding sets.

If $p$ is a periodic orbit of the geodesic flow on $T^1M$, we denote by $\ell(p)$ its period, and $\mathcal{L}_p$ the Lebesgue measure along $p$. 

\subsection{Pressure and Gibbs measures}

In this note, we consider equidistribution properties of periodic orbits of the geodesic flow towards Gibbs measures. Let us recall briefly the necessary background on these measures. 
We refer to \cite{PPS} and \cite{PS16} for more details. 

Let $F:T^1M \to\bbR$ be a H\"older continuous map, we still denote by $F$ its $\Gamma$-invariant lift to $T^1\widetilde M$. 
The following property, called {\em Bowen property} in many references as \cite{CT}, \cite{BCFT}, and {\em (HC)-type property} in \cite{BPP} is crucial in 
all estimates. It is a direct consequence of \cite[Lemma 3.2]{PPS}.  

\begin{lemm}\label{lemm:useful} Let $F:T^1\widetilde M\to \bbR$ be a H\"older map. 
For all $D>0$ and $x,y\in\widetilde M$, there exists $C>0$ depending on $D$, on the upperbound of the curvature, on the H\"older constants of $F$ and on $\sup_{B(x,D)}|F|$ and 
$\sup_{B(y,D)}|F|$  such that 
for all $x',y'$ in $\widetilde M$ with $d(x,x')\le D$, $d(y,y')\le D$, 
$$
\left|\int_x^y F-\int_{x'}^{y'}F\right|\le C\,.
$$
\end{lemm}

The series $\displaystyle \sum_{\gamma\in \Gamma}e^{\int_o^{\gamma o} F-sd(o,\gamma o)}$ has a critical exponent $\delta_F$. 
This exponent, when finite,  coincides with the pressure of $F$ when the manifold $M$ has pinched negative curvature and bounded derivatives of the curvature, see \cite{OP,PPS}.
 
By the obvious relation $\delta_{F+c}=\delta_F+c$ for any constant $c\in\bbR$, we can easily assume that $\delta_F>0$ as soon as it is finite. 

\medskip

A shadow $\mathcal{O}_x(B(y,R))$, for $x,y\in\widetilde M$, is the set of points $z\in\widetilde M\cup\partial\widetilde M$, such that the geodesic line from $x$ to $z$
intersects the ball $B(y,R)$.

The Patterson-Sullivan-Gibbs construction gives a measure $\nu_o^F$ on the boundary $\partial\widetilde M$, satisfying the following \emph{Sullivan Shadow Lemma}. It was first shown on hyperbolic manifolds for $F=0$ by Sullivan in \cite{Sull}, and is due to Mohsen \cite{Mohsen} for general potential when $\Gamma$ is cocompact. See \cite[lemma 3.10]{PPS} for a proof in general. 

\begin{lemm}[Shadow Lemma]\label{lemm:shadow} Let $F:T^1M\to \bbR$ be a H\"older-continuous map with finite critical exponent $\delta_F$, $\widetilde F:T^1\widetilde M\to \bbR$ its lift, 
 and $\nu_o^F$ be the measure on $\partial\widetilde M$ given by the Patterson-Sullivan-Gibbs construction. 
There exists $R>0$, such that for  all $r\ge R$, there exists  $C>0$ such that for all $\gamma \in \Gamma$, 
$$
\frac{1}{C}\,e^{\int_o^{\gamma o}(F-\delta_F)}\le \,\nu_o^F(\mathcal{O}_o(B(\gamma o,r))\le Ce^{\int_o^{\gamma o}(F-\delta_F)} \,.
$$
\end{lemm}

A nice consequence of the Shadow Lemma is the following proposition, that we will use in the proof of Theorem \ref{theo:equid} and which can be useful for other purposes . 

\begin{prop}\label{prop:nice}  With the above notations, for all $c>0$, there exists a constant $k>0$ such that for all $\alpha\in\Gamma$, all $r\ge R$ and all $T>0$, one has
$$
\sum_{\tiny \begin{array}{c}
\gamma\in \Gamma,\gamma o \in \mathcal{O}_o(B(\alpha o,r))\\
d(o,\gamma o)\in [T,T+c]\end{array}}e^{\int_o^{\gamma o} (F-\delta_F)} \le 
k e^{\int_o^{\alpha o} (F-\delta_F)}\,.
$$
The reverse inequality (with different constant) holds when $\Gamma$ acts cocompactly on $\widetilde M$.
\end{prop}

\begin{proof} By the Shadow Lemma (Lemma \ref{lemm:shadow}), the above sum is comparable, up to constants, to 
$$
\sum_{\tiny \begin{array}{c}
\gamma\in \Gamma,\gamma o \in \mathcal{O}_o(B(\alpha o,r))\\
d(o,\gamma o)\in [T,T+c]\end{array}}\nu_o^F(\mathcal{O}_o(B(\gamma o,r)).
$$
As $\Gamma$ acts properly on $\widetilde M$, the multiplicity of an intersection of such shadows is uniformly bounded. Therefore, the latter sum is comparable, up to 
constants, to 
$$
\nu_o^F\left(\bigcup_{\tiny \begin{array}{c}
\gamma\in \Gamma,\gamma o \in \mathcal{O}_o(B(\alpha o,r))\\
d(o,\gamma o)\in [T,T+c]\end{array}}\mathcal{O}_o(B(\gamma o,r))\right).
$$
As this union is included in $\mathcal{O}_o(B(\alpha o,2r))$, it is bounded from above by $\nu_o^F(\mathcal{O}_o(B(\alpha o,r)))$. 
A final application of the Shadow Lemma \ref{lemm:shadow} gives the desired upper bound. 

When $\Gamma$ acts cocompactly on $\widetilde M$, the above union covers $\mathcal{O}_o(B(\alpha o,r))$ so that, once again, the Shadow Lemma gives the desired lower bound. 
\end{proof}


\subsection{Finiteness criterion for Gibbs measures}

Through the Hopf coordinates, one defines a measure $\tilde m_F$ equivalent to $\nu_o^{\check{F}}\times \nu_o^F\times dt$ on $\partial^2\widetilde M\times\bbR\simeq T^1\widetilde M$, where $\check{F}(v):=F(-v)$, which is $\Gamma$-invariant and invariant under the geodesic flow; see \cite[Chapter 3]{PPS} for a precise construction.
The induced measure $m_F$ on the quotient, when finite, is the {\em Gibbs measure associated to $F$} involved in Theorem \ref{theo:equid}.

It is well known (Hopf-Tsuji-Sullivan-Gibbs Theorem) that $m_F$ is ergodic and conservative if and only if the series
$\displaystyle \sum_{\gamma\in \Gamma}e^{\int_o^{\gamma o} F-sd(o,\gamma o)}$ diverges at the critical exponent $\delta_F$, see \cite[Theorem 5.4]{PPS}. 

Let us recall the finiteness criterion shown in \cite{PS16}. For geometrically finite manifolds, it had been previously shown in \cite{DOP} for $F = 0$ and in \cite{Coudene, PPS} for general potentials.

If $W$ is a compact set of $\widetilde M$, we define $\Gamma_W$ as 
\begin{equation}
\Gamma_W=\{\gamma\in \Gamma,\exists x,y\in W, [x,\gamma y]\cap \Gamma W\subset W\cup \gamma W\}
\end{equation}

\begin{theo}[\cite{PS16},\cite{CDST}] \label{finiteness-criterion} Let $M$ be a negatively curved manifold with sectional curvature satisfying $\kappa\le -a^2<0$.
  Let $F:T^1M \to \bbR$ be a H\"older continuous map with finite critical exponent $\delta_F$. 
The measure $m_F$ is finite if and only if  it is ergodic and conservative, and there exists some compact set $W\subset \widetilde M$ whose interior intersects at least a closed geodesic, 
such that 
$$
\sum_{\gamma\in \Gamma_W} d(o,\gamma o)\,e^{\int_o^{\gamma o} (F-\delta_F)} <+\infty\,.
$$
\end{theo}

Note that in \cite{PS16} it is assumed that $M$ has pinched negative curvature $-b^2\le \kappa\le -a^2<0$, but the lower bound is not used in the proof of this finiteness criterion.


\subsection{Equidistribution w.r.t. the vague convergence}

There are many variants of equidistribution of weighted closed orbits w.r.t. the vague convergence, which are essentially all equivalent. See \cite[Chapter 9]{PPS} for several versions.

The statement which is the closest to our Theorem \ref{theo:equid} is the following.
\begin{theo}[Paulin-Policott-Schapira Thm 9.11 \cite{PPS}]\label{theo:22}
Let $M$ be a manifold with pinched negative curvature, whose geodesic flow is topologically mixing. 
Let $F:T^1M \to \bbR$ be a H\"older-continuous map with finite pressure $\delta_F$, which admits a finite equilibrium state $m_F$. 
 Assume without loss of generality that   $\delta_F$ is positive.

Then 
\begin{equation}\label{equation:equid-vague}
\delta_F T  e^{-\delta_F T} \sum_{p\in \mathcal{P}'(T)}e^{\int_p F} \frac{1}{\ell(p)}\mathcal{L}_p \to \frac{m_F}{\|m_F\|}\,,
\end{equation} in the vague topology, i.e. the dual of continuous maps with compact support on $T^1M$.  
\end{theo}

To get narrow equidistribution of periodic orbits, we will rather use the following statement. 

\begin{theo}[Paulin-Policott-Schapira thm 9.14  \cite{PPS}]\label{theo:equid-vague-annulus}
Let $M$ be a manifold with pinched negative curvature, whose geodesic flow is topologically mixing. 
Let $F:T^1M \to \bbR$ be a H\"older-continuous map with finite nonzero pressure $\delta_F$. Let $c>0$ be fixed. Assume   that $F$ 
admits a finite equilibrium state $m_F$. 
Then 
\begin{equation}\label{equation:equid-vague-annulus}
\frac{\delta_F T}{1-e^{-\delta_F c}}  T e^{-\delta_F T} \sum_{p\in \mathcal{P}'(T-c,T)}e^{\int_p F} \frac{1}{\ell(p)}\mathcal{L}_p \to \frac{m_F}{\|m_F\|}\,, \quad\mbox{and}
\end{equation} 
\begin{equation}\label{equation:equid-vague-annulus-bis}
\frac{\delta_F}{1-e^{-\delta_F c}}  e^{-\delta_F T} \sum_{p\in \mathcal{P}'(T-c,T)}e^{\int_p F}  \mathcal{L}_p \to \frac{m_F}{\|m_F\|}\,,
\end{equation}
in the vague topology, i.e. the dual of continuous maps with compact support on $T^1M$.  
\end{theo}


\section{Equidistribution in the narrow topology}


\subsection{An equidistribution statement on annuli}

Denote by $m_T$ the (locally finite and possibly infinite) measure 
$$
m_T=\frac{\delta_F T }{1-e^{-c\delta_F }} e^{-\delta_F T} \sum_{p\in \mathcal{P}'(T-c,T)}e^{\int_p F} \frac{1}{\ell(p)}\mathcal{L}_p 
$$
and by $m_{T,W}$ the (finite) measure  
$$
m_{T,W}=\frac{\delta_F T }{1-e^{-c\delta_F }} \sum_{p\in \mathcal{P}'_\mathcal{W}(T-c,T)}e^{\int_p F} \frac{1}{\ell(p)}\mathcal{L}_p 
$$

We will first prove the following theorem, and then deduce Theorem \ref{theo:equid} from it. 
\begin{theo}\label{theo:equid-annuli}
Under the assumptions of Theorem \ref{theo:equid}, the measures $m_{T,W}$ converge to $\bar m_F = \frac{m_F}{\|m_F\|}$ w.r.t. the narrow convergence, i.e. in the dual of bounded continuous maps on $T^1M$. 
\end{theo}

The proof of Theorem \ref{theo:equid-annuli} goes in two steps. 

First, we prove that $m_{T,W}-m_T$ goes to $0$ in the vague topology, so that by  Theorem \ref{theo:equid-vague-annulus},  
$m_{T,W}$ and $m_T$ have the same limit $\bar m_F$  in the vague topology. 

Second, we prove a tightness result\,: for all $\varepsilon>0$, there exists $K_\varepsilon\subset T^1M$ and $t_0>0$, such that for all $t\ge t_0$, 
$m_{T,W}(K_\varepsilon )\ge 1-\varepsilon$. 
 
Conclusion is then classical: we will deduce Theorem \ref{theo:equid} from Theorem \ref{theo:equid-annuli} in Paragraph \ref{final}. 

 
\subsection{Vague convergence }

Choose any compact set $W\subset M$ whose interior intersects a closed geodesic, and $\mathcal{W}=T^1W$. 
In this section, we prove the following.

\begin{prop}\label{prop:vague}
Under the assumptions of Theorem \ref{theo:equid}, $m_T-m_{T,W}$ converges to $0$ in the vague topology, when $T\to +\infty$. 
\end{prop}

\begin{proof} 
Let $\varphi\in C_c(T^1M)$ be a continuous compactly supported map. Without loss of generality, we can assume that $\|\varphi\|_\infty\le 1$
Choose $R>0$ such that the $R$-neighbourhood $W_R$ of $W$ in $M$ contains the projection on $M$ of the support of $\varphi$ in $T^1M$. 
Choose some $\varepsilon>0$ small enough so that the set $W_{-\varepsilon}$ of points of $W$ at distance at least $\varepsilon$ to the boundary is 
nonempty and intersects at least a closed geodesic.

Let $\widetilde W_R \supset \widetilde W\supset \widetilde W{-\varepsilon}$ 
be three compact sets of $\widetilde M$ which project respectively onto $W_R$,  $W$ and $W_{-\varepsilon}$.  
Choose a  point $o\in \widetilde W_{-\varepsilon}$ once for all.  

We begin with the following elementary inequality. For $p\in \mathcal{P}'(T-c,T)$, we have $|\delta_F T-\delta_F\ell(p)|\le c\delta_F $, so that  
\begin{equation}\label{eqn:elementarybound}
\left|m_{T,W_R}(\varphi)-m_{T,W}(\varphi)\right|\le \frac{\delta_F}{1-e^{-c\delta_F}} T e^{\delta_F c}
 \sum_{p\in \mathcal{P}'_{W_R}(T-c,T) \setminus \mathcal{P}'_W(T-c, T) }e^{\int_p (F-\delta_F)}\frac{\ell(p\cap W_R)}{\ell(p)}\,.
\end{equation}
Now, we will compare the latter sum with the sum appearing in Theorem \ref{finiteness-criterion}.

Given $p\in \mathcal{P}'_{W_R}(T-c,T) \setminus \mathcal{P}'_W(T-c, T)$, choose arbitrarily one isometry $\gamma_p\in \Gamma$, 
whose translation axis intersects $\widetilde W_R\setminus \widetilde W$
and projects on $M$ on the closed geodesic associated to the periodic orbit $p$, and whose translation length is $\ell(p)$.

Consider the geodesic from $o$ to $\gamma_p o$, parametrized as $(c(t))_{0\le t\le d(o,\gamma_p o)}$. 
It stays at bounded distance $\mbox{diam}(\tilde W_R)$ from the axis of $\gamma_p$. 
Therefore, these geodesics stay very close one from another, except at the beginning and at the end. 
More precisely, given any $\varepsilon>0$, there exists $\tau$ depending only 
on $\varepsilon$ and the upper bound of the curvature, 
such that for $t$ in  the interval $[\tau,d(o,\gamma_p o)-\tau]$, $c(t)$ is $\varepsilon$-close to 
the axis of $\gamma_p$. In particular, as this axis does not intersect $\Gamma.\widetilde W$, 
the geodesic $(c(t))_{\tau \le t \le d(o,\gamma_p o)-\tau}$ 
does not intersect $\widetilde W_{-\varepsilon}$, whereas the full segment $(c(t))_{0\le t\le d(o,\gamma_p o)}$ starts and ends in $\Gamma.o\subset \Gamma \widetilde W_{-\varepsilon}$. 

Denote by $x^-$ the last point of $c([0,\tau])$ (resp  $x^+$ the first point of $c([d(o,\gamma_p o)-\tau,d(o,\gamma_p o)])$) ) 
in $\Gamma \widetilde W_{-\varepsilon}$ and $\gamma^-$ (resp. $\gamma^+$) an element of $\Gamma$ such that $x^\pm\in\gamma^\pm \widetilde W_{-\varepsilon}$.
Observe that $d(o,\gamma^- o)\le \tau+\mbox{diam}(\widetilde W_R)$, and similarly $d(\gamma_p o,\gamma^+o)\le \tau + \mbox{diam}(\widetilde W_R)$.Moreover, by definition of $\Gamma_{\widetilde W_{-\varepsilon}}$, the element  $g_p:= (\gamma^-)^{-1}\circ \gamma^+$ belongs to $ \Gamma_{\widetilde W_{-\varepsilon}}$.
 
Using Lemma \ref{lemm:useful}, we see easily that there exists some constant $C$ depending on the upperbound of the curvature, 
on the H\"older constant of $F$ and $\|F_{|\mathcal{W}_R}\|_\infty$ and on the diameter of $\widetilde W_R$, such that 
\begin{equation}\label{eqn:comparaison}
\left|\ell(p)-d(o,g_p o) \right|\le C 
\quad\mbox{and}\quad 
\left|\int_p F-\int_o^{g_p o} F\right| \le C\,.
\end{equation}

 We have hence defined a procedure which, given any $p\in \mathcal{P}'_{W_R}(T-c,T) \setminus \mathcal{P}'_W(T-c, T)$ and 
any choice of an isometry $\gamma_p$ in the conjugacy class corresponding to $p$ whose axis intersects $\widetilde W_R$, determines a unique pair $(\gamma_p^-,g_p)$ where  
$d(o,\gamma^- o)\le \mbox{diam}(\widetilde W_R)+\tau$ and  $g_p= (\gamma^-)^{-1}\circ \gamma^+ $ is an element of $\Gamma_{\widetilde W_{-\varepsilon}}$, 
satisfying (\ref{eqn:comparaison}) and $d(\gamma^-g_po, \gamma_p o)\le \mbox{diam}(\widetilde W_R)+\tau$.  \\ 
Moreover, a coarse bound gives $\ell(p\cap\widetilde W_R)\le \ell(p)\le d(o,g_p o)+C$. 

\medskip

We want to control from above  (\ref{eqn:elementarybound}) by  a sum involving $ \Gamma_{\widetilde W_{-\varepsilon}}$. 
To do that, it is enough to control the multiplicity of the ``map'' $p\to g_p$.

Let $p_1$ be a periodic orbit leading to an element $g_p\in  \Gamma_{\widetilde W_{-\varepsilon}}$ by the above construction,
by some arbitrary choice of an axis of an isometry $\gamma_1$ intersecting $\widetilde W_R$.   
If another periodic orbit $p_2$ leads to the same element $g_p$, it means that there exists an isometry $\gamma_2$ and elements $\gamma_i^-\in \Gamma$, with 
$d(o,\gamma_i^-o)\le \mbox{diam}(\widetilde W_R)+\tau$, such that 
$$
d(o,(\gamma_2)^{-1}\gamma_2^-(\gamma_1^-)^{-1}\gamma_1 o)\le 2\mbox{diam}(\widetilde W_R)+2\tau\,.
$$
In particular, as $\Gamma$ is discrete, there are finitely many possibilities,   for $\gamma_i^-$ and therefore for $\gamma_2$, and $p_2$.\\
Denote by $N$ the maximal multiplicity of this map $p\to g_p$. 

Now, using (\ref{eqn:comparaison}), we bound from above the right hand side of (\ref{eqn:elementarybound}) by
$$
\frac{\delta_F}{1-e^{-c\delta_F}} T e^{c\delta_F } N\sum_{g\in \Gamma_{\widetilde W_{-\varepsilon}},\, d(o,go)\in [T-c-C,T+C]} e^{C+\delta_F}e^{\int_o^{g o} (F-\delta_F)}\frac{d(o,g o)}{T-c}\,.
 $$
This sum, up to constants, is bounded from above by 
$$
\sum_{g\in \Gamma_{\tilde W_{-\varepsilon}}, d(o,go)\ge T-c-C} d(o,go)\,e^{\int_o^{go} (F-\delta_F)}\,.
$$
Theorem \ref{finiteness-criterion} ensures us that this is the rest of a convergent series, whence it goes to zero as $T\to +\infty$. 
\end{proof}

\subsection{Tightness}
Let $W$ be a compact set of $M$, $W_R$ its $R$-neighbourhood, for $R$ large enough, and 
$\widetilde W\subset\widetilde W_R\subset \widetilde M$ compact sets which project onto $W\subset W_R$. Choose some
fixed point $o\in\widetilde W$. 
As above,  we denote by $\mathcal{W}$ the unit tangent bundle of $W$ and by abuse of notation, set $\mathcal{W}_R=T^1W_R$. 
\begin{prop} \label{prop:tight} Under the assumptions of Theorem \ref{theo:equid}, for all $\varepsilon>0$, there exists $R>0$ and $T>0$, such that for $t\ge T$, 
$$
m_{t,W}((\mathcal{W}_R)^c)\le \varepsilon\,.
$$
\end{prop}

\begin{proof} 
By definition of $m_{t,W}$, we have
\begin{equation}\label{eqn:first}
m_{t,W}((\mathcal{W}_R)^c)\le \frac{\delta_F}{1-e^{-c\delta_F}} \frac{t}{t-c}\sum_{p\in\mathcal{P}'_W(t-c,t)}e^{\int_p (F-\delta_F)}\ell(p\cap W_R^c)\,.
\end{equation}
If $p$ is a periodic orbit appearing in the above sum with $\ell(p\cap W_R^c)\neq 0$, there exists an hyperbolic isometry
$\gamma_p\in \Gamma$ whose axis projects onto the closed geodesic associated to $p$, which intersects $\widetilde W$ and $\gamma_p\widetilde W$, but also 
$\Gamma.(\widetilde W_R)^c$. 
Denote by 
$\Gamma_{\widetilde W}(p,W_R)$ the set of elements $\alpha\in \Gamma_{\widetilde W}$ such that some axis of some isometry $\gamma_p$ associated to $p$ as above intersects
$\widetilde W$, and $\alpha\widetilde W$ and goes outside $\Gamma \widetilde W_R$ between $\widetilde W$ and $\alpha\widetilde W$. 
Each $\alpha \in \Gamma_{\widetilde W}(p,W_R)$ encodes exactly one excursion of the periodic orbit $p$ outside $W_R$. 
In particular, $d(o,\alpha o)\ge 2R$, and we have
$$
\ell(p\cap (W_R)^c)\le 
\sum_{\alpha\in \Gamma_{\widetilde W}(p,W_R), d(o,\alpha o)\ge 2R} \left(d(o,\alpha o)+2\mbox{diam}(\widetilde W)\right)\,.
$$
We deduce that (\ref{eqn:first}) is bounded from above, up to some constants, by
$$
\sum_{p\in\mathcal{P}'_W(t-c,t)}e^{\int_p (F-\delta_F)} \sum_{\alpha\in \Gamma_{\widetilde W}(p,W_R), d(o,\alpha o)\ge 2R} (d(o,\alpha o) + 2 \mbox{diam}(\widetilde W))
$$ 
Observe now that if $\alpha\in \Gamma_{\widetilde W}(p,W_R)$ and $\gamma_p$ is an isometry whose axis  intersects
$\widetilde W$, and $\alpha\widetilde W$, then $\gamma_p o$ belongs to  the shadow $\mathcal{O}_o(B(\alpha o, r))$ for $r=\mbox{diam}(\tilde W)$. 
As in (\ref{eqn:comparaison}) in the proof of Proposition \ref{prop:vague}, we know that $\int_p(F-\delta_F)$ is uniformly close to $\int_o^{\gamma_p o}(F-\delta_F)$, 
which, by the Shadow Lemma \ref{lemm:shadow}, is comparable to $\nu_o^F(\mathcal{O}_o(B(\gamma o,r))$ for $r$ large enough.
Up to some constants, the above sum is bounded from above by
$$
\sum_{\alpha\in\Gamma_{W},\,d(o,\alpha o)\ge 2R} d(o,\alpha o)\sum_{\gamma \in \Gamma, t-c-C \le d(o,\gamma o)\le t+C} {\bf 1}_{\mathcal{O}_o(B(\alpha o,r)}(\gamma o)\nu_o^F(\mathcal{O}_o(B(\gamma o,r)),.
$$
By Proposition \ref{prop:nice}, (\ref{eqn:first}) is then dominated (up to multiplicative constants) by

$$
\sum_{\alpha\in\Gamma_W,d(o,\alpha o)\ge 2R} d(o,\alpha o) e^{\int_o^{\alpha o}(F-\delta_F)}.
$$
This is the rest of the convergent series appearing in Theorem \ref{finiteness-criterion}. 
Therefore, it goes to $0$ when $R\to +\infty$, so that for $R$ large enough, it is smaller than $\varepsilon$. 
It is the desired result.  
\end{proof}

\subsection{Conclusion}

\subsubsection{Proof of Theorem \ref{theo:equid-annuli}}

Getting narrow convergence from vague convergence and tightness is very classical, we recall it for the comfort of the reader. 

By \cite[Theorem 9.14]{PPS} (see Theorem \ref{theo:equid-vague-annulus}), we know that $m_{t}$ converges towards the normalized probability measure 
$\overline{m}_F:=\frac{m_F}{\|m_F\|}$ in the vague topology, i.e. the dual of $C_c(T^1 M)$. 
Proposition \ref{prop:vague} ensures that $m_{t,W}$ also converges vaguely to   $\overline{m_F}$.

Given a continuous bounded function $\varphi$ and $\varepsilon>0$, by Proposition \ref{prop:tight}, one can find a compact 
set $\mathcal{W}_R$ such that $m_{t,W}((\mathcal{W}_R)^c)\le \varepsilon$ for all $t\ge T$, 
$\overline{m_F}(\mathcal{W}_R)\ge 1-\varepsilon$, 
and $\left|\int_{\mathcal{W}_R^c}\varphi\,d\overline{m}_F\right|\le \varepsilon$. 
Choose $\psi\in C_c(T^1M)$ with $\psi=\varphi\in W_R$, and $|\psi|\le\|\varphi\|$, and $
\left|\int (\varphi-\psi)\,d\overline{m}_F\right|\le 2\varepsilon$. 

By the above choices, 
$$
m_{t,W}(\varphi)-m_{t,W}(\psi)|\le \|\varphi\|_\infty m_{t,}(W_R^c)\le \varepsilon\|\varphi\|_\infty\,.
$$
By Proposition \ref{prop:vague},  $m_{t,W}(\psi)\to \int \psi\,d\overline{m}_F$, which is $2\varepsilon$-close to $\int\varphi \,d\overline{m}_F$. 
The result follows.

\subsubsection{Proof of Theorem \ref{theo:equid}}\label{final}

The proof is elementary and similar to the deduction of \cite[Thm 9.14]{PPS} (see Theorem \ref{theo:equid-vague-annulus}) from \cite[Thm 9.11]{PPS} (see Theorem \ref{theo:22}), 
but in the other direction. 
Let us begin with an elementary  lemma, which is a reformulation of \cite[Lemma 9.5]{PPS}. 

\begin{lemm} Let $I$ be a discrete set and $f,g:I\to [0,+\infty[$ be maps with $f$ proper. For all $c>0$ and 
$\delta,\kappa\in\bbR$, with $\delta+\kappa>0$ , the following are equivalent:
\begin{enumerate}
\item
  as $t\to +\infty$, $\displaystyle 
\sum_{i\in I, t-c<f(i)\le t} g(i) \sim \frac{1-e^{-c\delta}}{\delta}e^{\delta t}\,;
$
\item  as $t\to +\infty$, 
$
\displaystyle\sum_{i\in I,f(i)\le t}e^{\kappa f(i)} g(i)\sim \frac{e^{(\delta+\kappa)t}}{\delta+\kappa}\,.
$
\end{enumerate}
\end{lemm}
\begin{proof} Note that eventhough only one implication of the above lemma is stated in \cite[Lemma 9.5]{PPS}, its proof gives indeed the equivalence. We refer the reader to \cite[p. 182]{PPS} for details.  
\end{proof}

Now,  let us conclude the proof of  Theorem \ref{theo:equid}. By linearity, it is 
enough to prove it for nonnegative maps $\varphi$ that satisfy $ \int\varphi d\overline{m}_F\neq 0$. 
The desired result follows then from Theorem \ref{theo:equid-annuli} and applying the above lemma with $I=\mathcal{P}'_\mathcal{W}$, and for $p\in I=\mathcal{P}'_\mathcal{W}$, 
$f(p)=\ell(p)$ and $\displaystyle g(p)=e^{\int_p F}\frac{\int_p \varphi}{\int\varphi d\overline{m}_F}$. .


\section{Narrow equidistribution on $CAT(-1)$ metric spaces}

The proof of Theorem \ref{theo:CAT(-1)} is exactly the same as the above proof in the Riemannian case. Nevertheless, in this generality, the basic ingredients which we use (vague equidistribution for periodic orbits and finiteness criterion for Gibbs measure) are only known for the potential $F = 0$. We just mention in this section which parts of our proof have to be adapted, and how. 

First, the definition of the geodesic flow and its invariant measures is now well-known, with many properties established by  Roblin
\cite{Roblin}. 

The vague equidistribution result which we use, Theorem \ref{theo:equid-vague-annulus} above (Theorem 9.14 of \cite{PPS}), 
had been established earlier in the case $F=0$ in the $CAT(-1)$-setting in \cite[Thm 9.1.1]{Roblin}.

The finiteness criterion for Gibbs measures, Theorem \ref{finiteness-criterion} above (\cite{PS16}), has been extended in 
\cite[Theorem 4.16]{CDST} to the Gromov-hyperbolic setting (which includes $CAT(-1)$ spaces) when $F=0$.

Shadow Lemma \ref{lemm:shadow}, Proposition \ref{prop:nice} and Lemma \ref{lemm:useful} all hold without difficulty in the $CAT(-1)$-setting 
(see  \cite{Roblin} when $F=0$ and \cite{BPP} for general potentials).
The arguments of the proofs of Propositions \ref{prop:vague} and \ref{prop:tight} do not use the Riemannian structure, and hold in the $CAT(-1)$-setting. 

Theorem \ref{theo:CAT(-1)} follows.

\begin{rema}\rm 
Our main result probably holds for a general H\"older-continuous potential  in the $CAT(-1)$-setting. We do not prove it in this generality here, because vague equidistribution for periodic orbits and finiteness criterion for Gibbs measure, although likely to be true, would be very long to check.  

Let us mention what would be the ingredients of a proof.

\begin{itemize}
\item In the $CAT(-1)$-setting, thermodynamical formalism with nonzero potentials has already been considered. The construction of Gibbs measures and the variational principle are proven in Roblin \cite{Roblin} when $F=0$ and Broise-Parkonnen-Paulin \cite{BPP} for general H\"older-continuous potentials. 
\item The finiteness criterion for Gibbs measures of Pit-Schapira \cite{PS16} is extended in \cite{CDST} in the Gromov-hyperbolic setting for the potential $F=0$ and should probably
hold also for general potentials in the $CAT(-1)$-setting. (Note that thermodynamical formalism with nonzero potentials is possible on $CAT(-1)$-spaces, but more delicate in general Gromov-hyperbolic spaces.)
\item The equidistribution theorem of periodic orbits w.r.t. the vague topology is proven in \cite{Roblin} in the $CAT(-1)$-setting for $F=0$, in \cite{PPS} for Riemannian negatively curved manifodlds, and in \cite{BPP}
on real trees. It seems that the restriction to manifolds and trees in the equidistribution theorem \cite[Thm. 11.1]{BPP} was only motivated by the applications that the authors had in mind. 
As they were  dealing with thermodynamical formalism in the $CAT(-1)$-setting, it is likely that this vague equidistribution theorem should also hold on $CAT(-1)$-spaces. 
\end{itemize}

\end{rema}

\addcontentsline{toc}{section}{References}
\bibliography{biblio}
\bibliographystyle{amsalpha}
\end{document}